\documentclass[12pt]{article}
\usepackage{amsmath, amssymb}
\usepackage{amsthm}
\usepackage{graphicx}
\usepackage{hyperref}
\usepackage[T1]{fontenc}
\usepackage{listings}
\usepackage[all]{xy}
\usepackage{graphicx}
\graphicspath{ {./Images/} }

\newcommand{\Set}{\mathbf{Set}}

\newtheorem{proposition}{Proposition}
\newtheorem{remark}{Remark}

\newtheorem{example}{Example}

\title{On the zero-classes of monoid semi-congruences}
\author{M. Hoefnagel, N. Martins-Ferreira and M. Sobral\\
}

\begin{document}

\maketitle

\begin{abstract}

This paper studies the zero-classes of monoid semi-congruences, understood as internal reflexive relations on a monoid. Classical examples include normal submonoids, which arise as zero-classes of congruences, and positive cones, which are the zero-classes of preorders; both admit well-known syntactic characterizations via the Eilenberg syntactic equivalence and the syntactic preorder introduced by Pin, respectively. Beyond these cases, however, no general notion of a syntactic object characterizing zero-classes of semi-congruences, also called clots, has been established.

We address this gap by introducing a syntactic relation that is reflexive and characterizes clots whenever it is compatible with the monoid operation, a property that is not automatic in contrast to the congruence and preorder settings. We further develop a hierarchy of conditions on submonoids of a given monoid that induce syntactic relations with progressively stronger properties: first ensuring compatibility with the monoid operation, and subsequently enforcing additional relational properties such as transitivity and symmetry. Several illustrative situations are discussed, including the cases of positive cones and normal submonoids.

\end{abstract}
\noindent
\textbf{keywords:} clot, syntactic congruence, semi-congruence, monoid. 

\section{Introduction}

For a submonoid $M$ of a monoid $A$, slightly generalizing the notion of {\em syntactic equivalence} introduced by S.~Eilenberg \cite{E} and that of {\em syntactic preorder} due to J.-E.~Pin \cite{P,HP}, we introduced in \cite{MFS3} the notions of {\em $M$-equivalence relation}, which is always a congruence, and of {\em $M$-preorder relation}, which is also always compatible with the monoid operation, that is, both are internal relations.

In the same work, we associated to each submonoid $M$ a reflexive relation $R_M$ on $A$, called the {\em $M$-reflexive relation} whenever $R_M$ is internal. The fact that this relation is not always internal was proved by the first author, and a corresponding counterexample is presented below. Indeed, this already follows from an example given in \cite{MFS3}, which was previously overlooked, as explained before Example~2.

{\em Clots} are the zero-classes of internal reflexive relations. Their first categorical definition was presented in \cite{JMU} for semi-abelian categories and later in \cite{M} for regular categories with finite coproducts, where internal reflexive relations are called {\em semi-congruences}.

\section{Preliminaries}

The {\em syntactic equivalence} of Eilenberg \cite{E} and the {\em syntactic preorder} of J.-E.~Pin \cite{P,HP} were originally defined for pairs $(A,M)$ where $A$ is a free monoid and $M$ is a subset of $A$. We slightly generalize these notions by considering, for a subset $M$ of an arbitrary monoid $A$, the corresponding {\em $M$-congruence} and {\em $M$-preorder}, and by introducing, for such pairs, the notion of an {\em $M$-relation}.

Let $A$ be a monoid and let $M$ be a subset of $A$, not necessarily a submonoid. We define on $A$ an $M$-congruence, an $M$-preorder, and an $M$-relation as follows:
\begin{enumerate}
\item The $M$-congruence is denoted by $\sim_{M}$. An element $a\in A$ is said to be $M$-congruent to $b\in A$, written $a\sim_{M}b$, if and only if, for every $x,y\in A$,
\[
xay \in M \;\Leftrightarrow\; xby \in M.
\]

\item The $M$-preorder is denoted by $\leq_{M}$. An element $a\in A$ is said to be $M$-preceded by $b\in A$, written $a\leq_{M}b$, if and only if, for every $x,y\in A$,
\[
xay \in M \;\Rightarrow\; xby \in M.
\]

\item The $M$-relation is denoted by $R_{M}$. An element $a\in A$ is said to be $M$-related to $b\in A$, written $aR_{M}b$, if and only if, for every $x,y\in A$,
\[
(xay = 1) \;\Rightarrow\; xby \in M.
\]
\end{enumerate}

\begin{proposition}
Let $M$ be a subset of a monoid $A$.
\begin{enumerate}
\item The relation $\sim_{M}$ is an internal equivalence relation in the category of monoids.
\item The relation $\leq_{M}$ is an internal preorder in the category of monoids.
\item The relation $R_{M}$ is reflexive if and only if $1\in M$.
\item If $1\in M$, then for every $a,b,c\in A$, $aR_{M}b$ implies $acR_{M}bc$ and $caR_{M}cb$.
\item The relation $R_M$ is an internal relation in the category of monoids if and only if, for every $a,a',b,b'\in A$,
\[
\forall x_1,y_1\in A,\ (x_1aa'y_1=1)\Rightarrow x_1bb'y_1\in M
\]
holds whenever
\[
\left\lbrace
\begin{array}{l}
\forall x_2,y_2\in A,\ (x_2ay_2=1)\Rightarrow x_2by_2\in M, \\[2pt]
\forall x_3,y_3\in A,\ (x_3a'y_3=1)\Rightarrow x_3b'y_3\in M.
\end{array}
\right.
\]
\end{enumerate}
\end{proposition}

Let $S\subseteq A\times A$ be a relation on $A$. By the \emph{zero-class} of $S$ we mean the set
\[
\{u\in A \mid 1Su\},
\]
which will be denoted by $[1]_S$.

From now on, we restrict to the case where $M$ is a submonoid of $A$. We say that $M$ is a \emph{normal submonoid} of $A$ if there exists a congruence (that is, an internal equivalence relation) $E\subseteq A\times A$ such that $M=[1]_E$. Similarly, we say that $M$ is a \emph{positive cone} of $A$ if there exists an internal preorder (that is, an internal reflexive and transitive relation) $P\subseteq A\times A$ such that $M=[1]_P$. Finally, we say that $M$ is a \emph{clot} of $A$ if there exists an internal reflexive relation $S\subseteq A\times A$ such that $M=[1]_S$.

\begin{proposition}
Let $M$ be a submonoid of the monoid $A$.
\begin{enumerate}
\item The submonoid $M$ is a normal submonoid of $A$ if and only if $M$ is the zero-class of $\sim_{M}$.
\item The submonoid $M$ is a positive cone of $A$ if and only if $M$ is the zero-class of $\leq_{M}$.
\item If the submonoid $M$ is a clot of $A$, then $M=[1]_{R_M}$.
\item If $R_M$ is an internal relation and $M=[1]_{R_M}$, then $M$ is a clot of $A$.
\item The condition $M=[1]_{R_M}$ is equivalent to the following: for every $x,y\in A$ and every $u\in M$,
\[
xy=1 \;\Rightarrow\; xuy \in M.
\]
\item If $A$ is a group, then $M=[1]_{R_M}$ if and only if $M$ is closed under conjugation in $A$.
\end{enumerate}
\end{proposition}

In what follows, we analyze the $M$-relation in more detail.


\section{Semi-congruences of monoids}

In \cite{MFS3}, it was shown that for any submonoid $M$ of a monoid $A$, the relations $\sim_M$ and $\leq_M$ introduced above are, respectively, an internal equivalence relation (that is, a congruence) and an internal preorder on $A$. Moreover, they provide intrinsic characterizations of normal submonoids and positive cones, as recalled in the proposition above. 

Whenever $R_M$ is an internal relation, we call it the \emph{$M$-reflexive relation} and, in this case, the condition  $M = [1]_{R_M}$ is a characterization for $M$ being a clot as we recall next (\cite{MFS3}, Proposition 3).

\begin{proposition}
 Let $M$ be a submonoid of a monoid $A$. Then $M$ is the zero-class of the reflexive relation $R_M$ if and only if the following condition holds:
\begin{quote}
(R) For every $x,y\in A$ and every $u\in M$,
\[
xy = 1 \;\Rightarrow\; xuy \in M.
\]
\end{quote}
If $M$ is a clot, then $M = [1]_{R_M}$, and the converse holds provided that the relation $R_M$ is internal.
\end{proposition}

\begin{proof}
We first show that $[1]_{R_M} \subseteq M$. Indeed, if $1 R_M u$, then for every $x,y\in A$ we have
\[
x1y = 1 \;\Rightarrow\; xuy \in M.
\]
In particular, taking $x=y=1$, we conclude that $u\in M$.

Conversely, $M \subseteq [1]_{R_M}$ if and only if condition (R) is satisfied. Consequently, under condition (R), we have
$M = [1]_{R_M}$.

Now assume that $M$ is a clot, that is, the zero-class of some internal reflexive relation $S$ on $A$. If $1Su$, then by reflexivity and compatibility of $S$ we have $xySxuy$. Thus, whenever $xy=1$, we obtain $1Sxuy$, that is, $xuy\in M$. Hence condition (R) holds and therefore $M=[1]_{R_M}$.

If $R_M$ is compatible, then the converse implication holds.
\end{proof}

The proof that the relation $R_M$ is not always a semi-congruence was given by the first author by constructing the following example.

\begin{example} \label{example: bicyclic}\em
Recall that the {\em bicyclic} monoid $B$ is the quotient of the free monoid on two generators $x,y$ by the congruence generated by the single relation $xy\sim 1$. Every element of $B$ is uniquely representable as $y^n x^m$ for some natural numbers $n,m$. Consider the submonoid
\[
M=\{\,y^n x^m \mid n,m\in 2\mathbf{N}\,\}
\]
of $B$.

We show that $R_M$ is not an internal relation. First, we have $(y^2x) R_M (yx^2)$, since for any elements $y^n x^m$ and $y^{n'} x^{m'}$ of $B$ such that
\[
(y^n x^m)(y^2x)(y^{n'} x^{m'})=1,
\]
it follows that $n=0=m'$ and $m=n'+1$. Hence
\[
x^{n'+1}(yx^2)y^{n'}=x^2\in M.
\]

Similarly, we have $x R_M y$, since $(y^n x^m)x(y^{n'} x^{m'})=1$ implies $n=0=m'$ and $m+1=n'$, from which it follows that
\[
x^m y y^{m+1}=y^2\in M.
\]

However, their product does not satisfy the relation: we do not have $(yx) R_M (y^2x^2)$, since
\[
x(y^2x^2)y=yx\notin M.
\]
And so the relation is not internal.
\end{example}

Another way to see that $R_M$ is not always internal is suggested by Proposition~2 (item~(3)), where the condition $M = [1]_{R_M}$ is shown to be necessary for $R_M$ to be a semi-congruence. Contrary to what we stated in Example~(J) of \cite{MFS3}, which we recall below, that example exhibits a submonoid $M$ for which $[1]_{R_M} \neq M$ and consequently  the associated relation $R_M$ is not internal.


\begin{example}\em (Example (J), \cite{MFS3})
Let $A=\Set(\mathbf{N},\mathbf{N})$ be the monoid of all endofunctions of the natural numbers, with composition, and let $M$ be the submonoid generated by the function $u=2\times -$. Then $M$ consists of the functions $u^n=2^n\times -$ for $n\in\mathbf{N}_0$, and $u^n$, for $n>0$, does not belong to the zero-class of the relation $R_M$.

Indeed, let $f,g\in A$ be defined by $g(x)=x+1$ and by $f(x)=x-1$ if $x>1$, and $f(1)=1$. Then $f\cdot g=1_{\mathbf{N}}$, but
$f\cdot u^n\cdot g\notin M$, since
\[
f\cdot u^n\cdot g(x)=2^n(x+1)-1
\]
for $x\in\mathbf{N}$. Thus $[1]_{R_M}=\{1_{\mathbf{N}}\}$ and $R_M$ is not a semi-congruence.
\end{example}

For the trivial submonoids $M=\{1_A\}$ and $M=A$, the relation $R_M$ is always a semi-congruence on $A$.

There is a large class of monoids for which these reflexive relations are semi-congruences for every submonoid $M$. Indeed, the relation $R_M$ is compatible with the monoid operation as soon as the following condition holds:
\begin{equation}
(*)\qquad
\forall x,y,s,t\in A,\ (xy=1,\ xs\in M,\ ty\in M)\Rightarrow ts\in M.
\end{equation}

This condition follows from the requirement that every left (or right) invertible element of $A$ is invertible, namely,
\[
\forall x,y\in A,\ (xy=1\Rightarrow yx=1),
\]
which characterizes Dedekind finite monoids, also called von Neumann finite or directly finite monoids.

In Dedekind finite monoids, condition $(*)$ holds: from $xs,ty\in M$ we obtain $tyxs\in M$, and since $yx=1$ whenever $xy=1$, we conclude that $ts\in M$.

Dedekind finite monoids include groups, commutative monoids, finite monoids, and cancellative monoids, but not all monoids. For example, the monoid $\Set(X,X)$ for any infinite set $X$, as well as the bicyclic monoid, are not Dedekind finite.

The condition of being Dedekind finite is stronger than condition $(*)$, and condition $(*)$ is not necessary for the compatibility of $R_M$, as shown by Example~(I) in \cite{MFS3}.

\begin{remark}\em
In the category of groups, if $M$ is a subgroup of a group $A$, condition (R) means that $M$ is a normal subgroup of $A$. In this case, the congruence $R_M$ coincides with the classical equivalence relation induced by $M$,
\[
a\sim b \;\Leftrightarrow\; ab^{-1}\in M,
\]
which is a congruence if and only if $M$ is normal.

If $A$ is a group and $M$ is a submonoid of $A$, condition (R) means that $M$ is closed under conjugation. Such pairs $(A,M)$ form the objects of a category equivalent to that of preordered groups, by defining a preorder on $A$ by
\[
a\leq b \;\Leftrightarrow\; ba^{-1}\in M,
\]
that is, $b\in Ma=aM$, as described in \cite{CMFM}.
\end{remark}

\section{When is $R_M$ a semi-congruence?}

To answer this question we consider a more general categorical context.

Let $\mathcal{C}$ be the category whose objects are pairs $(A,M)$, where $A$ is a monoid and $M$ is a submonoid of $A$, and whose morphisms
$f\colon (A,M)\to (A',M')$ are monoid homomorphisms $f\colon A\to A'$ such that $f(M)\subseteq M'$. We consider the following chain of full subcategories of $\mathcal{C}$:
\begin{equation}
\mathcal{C}_5 \subset \mathcal{C}_4 \subset \mathcal{C}_3 \subset \mathcal{C}_2 \subset \mathcal{C}_1 \subset \mathcal{C},
\end{equation}
where:
\begin{enumerate}
\item $\mathcal{C}_1$ consists of all pairs $(A,M)$ such that $R_M$ is an internal relation;
\item $\mathcal{C}_2$ consists of all pairs $(A,M)$ satisfying condition $(*)$:
\[
\forall x,y,s,t\in A,\ (xy=1,\ xs\in M,\ ty\in M)\Rightarrow ts\in M;
\]
\item $\mathcal{C}_3$ consists of all pairs $(A,M)$ such that $A$ is Dedekind finite;
\item $\mathcal{C}_4$ consists of all pairs $(A,M)$ such that $A$ is a group;
\item $\mathcal{C}_5$ consists of all pairs $(A,M)$ such that both $A$ and $M$ are groups.
\end{enumerate}

All these inclusions are strict:
\begin{enumerate}
\item $\mathcal{C}_1 \neq \mathcal{C}$, by Examples~1 and~2;
\item $\mathcal{C}_2 \neq \mathcal{C}_1$, as shown by the monoid $A=\Set(\mathbf{N},\mathbf{N})$ under composition with $M=\{1_{\mathbf{N}}\}$ (see \cite{MFS3}, Example~I(ii));
\item $\mathcal{C}_3 \neq \mathcal{C}_2$, as shown by $A=\Set(\mathbf{N},\mathbf{N})$ with $M=A$ (see \cite{MFS3}, Example~I(i));
\item $\mathcal{C}_4 \neq \mathcal{C}_3$, since not every Dedekind finite monoid is a group;
\item $\mathcal{C}_5 \neq \mathcal{C}_4$, since in $\mathcal{C}_5$ both $A$ and $M$ are required to be groups.
\end{enumerate}

In order to study conditions ensuring that $M$ is the zero-class of $R_M$, we introduce the following full subcategories of $\mathcal{C}$:
\begin{enumerate}
\item $\mathcal{C}_0$, consisting of all pairs $(A,M)$ such that $M=[1]_{R_M}$;
\item $\mathcal{C}_{0.5}$, consisting of all pairs $(A,M)$ such that $M$ is a clot;
\item $\mathcal{C}_{(i,0)}=\mathcal{C}_i\cap \mathcal{C}_0$, for $i=1,\dots,5$.
\end{enumerate}

These give rise to another chain of full subcategories:
\begin{equation}
\mathcal{C}_{(1,0)} \subset \mathcal{C}_{0.5} \subset \mathcal{C}_0 \subset \mathcal{C}.
\end{equation}

The inclusions satisfy the following properties:
\begin{enumerate}
\item $\mathcal{C}_0\neq\mathcal{C}$, since there are examples for which $[1]_{R_M}\neq M$, such as Example~(J) in \cite{MFS3};
\item $\mathcal{C}_{0.5}\subset \mathcal{C}_0$, since clots are submonoids $M$ satisfying (\cite{MFS3} Thm 4(f))
\begin{eqnarray}\label{eq: clot}
\forall n\in\mathbf{N},\ \forall a_1,\ldots,a_{n+1}\in A,\ \forall u_1,\ldots,u_n\in M,\nonumber\\
a_1\cdots a_{n+1}=1 \Rightarrow a_1u_1a_2\cdots a_nu_na_{n+1}\in M,
\end{eqnarray}
which implies
\begin{equation}\label{eq: M=[0]R_M}
\forall x,y\in A,\ \forall u\in M,\ xy=1 \Rightarrow xuy\in M,
\end{equation}
that is, $M=[1]_{R_M}$;
\item It is not known whether the inclusion $\mathcal{C}_{(1,0)}\subset \mathcal{C}_{0.5}$ is strict, that is, whether there exists a submonoid $M$ of a monoid $A$ such that
$M=[1]_S$ for some semi-congruence $S$ (and hence $M=[1]_{R_M}$) while $R_M$ is not internal.
\end{enumerate}

Summarizing, we obtain the following diagram of full subcategories of $\mathcal{C}$:
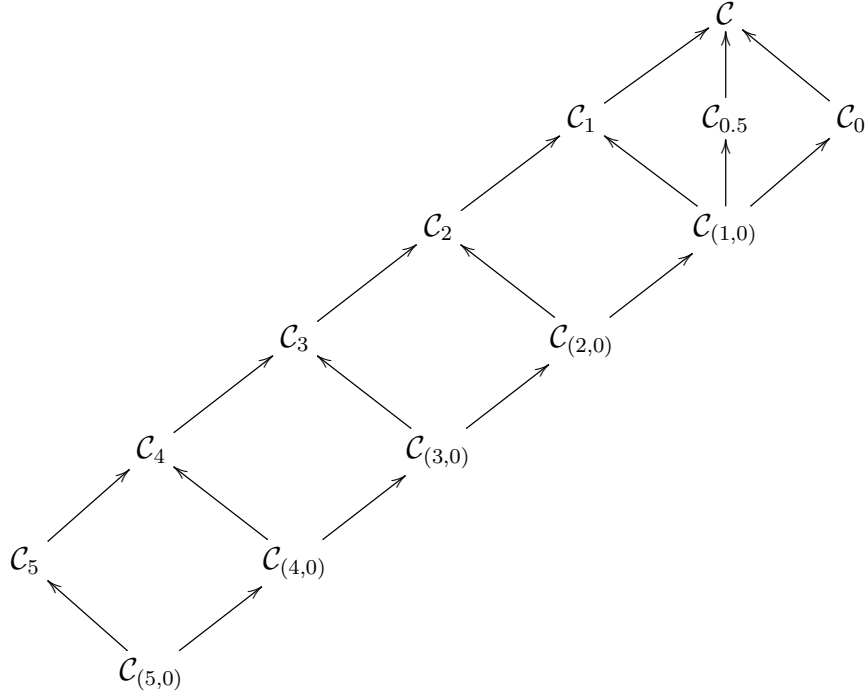
\begin{align*}
\xymatrix{&&&&& \mathcal{C}\\
&&&& \mathcal{C}_{1} \ar[ur] & \mathcal{C}_{0.5}\ar[u] & \mathcal{C}_{0}\ar[ul]\\
&&& \mathcal{C}_{2}\ar[ur] && \mathcal{C}_{(1,0)}\ar[ul] \ar[u]\ar[ur]\\
&&\mathcal{C}_{3}\ar[ur] && \mathcal{C}_{(2,0)}\ar[ul] \ar[ur]\\
&\mathcal{C}_{4}\ar[ur] && \mathcal{C}_{(3,0)}\ar[ul] \ar[ur]\\
\mathcal{C}_{5}\ar[ur] && \mathcal{C}_{(4,0)}\ar[ul] \ar[ur]\\
 & \mathcal{C}_{(5,0)}\ar[ul] \ar[ur]
}
\end{align*}

Here $\mathcal{C}_{(1,0)}$ is the category of all pairs $(A,M)$ such that $R_M$ is internal and $M$ is the zero-class of $R_M$. In particular, $\mathcal{C}_{(4,0)}$ is equivalent to the category of preordered groups, and $\mathcal{C}_{(5,0)}$, whose objects are pairs $(A,M)$ with $M$ a normal subgroup of the group $A$, is equivalent to the category whose objects are congruences in groups.

We now turn to preordered monoids, that is, monoids equipped with an internal reflexive and transitive relation.

\begin{enumerate}
\item Let $\mathcal{D}$ be the full subcategory of $\mathcal{C}$ consisting of all pairs $(A,M)$ such that $M=[1]_{\leq_M}$, that is,
\[
a\in M \;\Longleftrightarrow\; \forall x,y\in A,\ (xy\in M \Rightarrow xay\in M).
\]
By Proposition~2 of \cite{MFS3}, $M$ is the positive cone of some internal preorder if and only if it is the zero-class of the $M$-preorder. Thus, $\mathcal{D}$ may be regarded as the category of preordered monoids determined by their positive cones. Unlike $\mathcal{C}_{(4,0)}$, which is equivalent to the category of preordered groups, $\mathcal{D}$ is strictly smaller than the category of all preordered monoids, since not every preorder is determined by its positive cone (\cite{MFS2}, Example~1).

\item Let $\mathcal{D}_r$ be the full subcategory of $\mathcal{D}$ consisting of all pairs $(A,M)$ such that $M$ is {\em right homogeneous}, that is, $aM\subseteq Ma$ for all $a\in A$. Equivalently, the preorder defined by $a\leq_r b$ if and only if $b\in Ma$ is internal (\cite{MFS2}, Proposition~2), and $M$ is its positive cone.

\item In a similar way, let $\mathcal{D}_l$ be the full subcategory of $\mathcal{D}$ consisting of all pairs $(A,M)$ such that $M$ is {\em left homogeneous}, that is, $Ma\subseteq aM$ for all $a\in A$. Equivalently, the preorder defined by $a\leq_l b$ whenever $b\in aM$ is internal, with $M$ as its positive cone.
\end{enumerate}

In this situation we have strict inclusions
\[
\mathcal{C}_{(4,0)} \subset \mathcal{D}_i \subset \mathcal{D} \subset \mathcal{C}_{0.5} \subset \mathcal{C},
\qquad i\in\{r,l\}.
\]

Indeed:
\begin{enumerate}
\item $\mathcal{D}\neq \mathcal{C}_{0.5}$;
\item $\mathcal{D}_i\neq \mathcal{D}$, since the $M$-preorder strictly contains the preorders $a\leq b$ if $b\in Ma$ or $b\in aM$;
\item $\mathcal{C}_{(4,0)}\neq \mathcal{D}_i$, since in $\mathcal{C}_{(4,0)}$ the monoid $A$ is a group and $M$, being closed under conjugation, is both left and right homogeneous.
\end{enumerate}

We say that $M$ is {\em homogeneous} if it is both left and right homogeneous.

In \cite{MFS2}, the term {\em normal submonoid} was used for what we now call homogeneous submonoids (and {\em left normal} or {\em right normal} for the one-sided notions). This terminology was inappropriate, since normal submonoids in our present sense need not be homogeneous, as shown by the following example.

\begin{example}\em
Let $A=\Set(X,X)$ be the monoid of endofunctions of a finite set $X$ with at least two elements, under composition, and let $M=S_X$ be the submonoid of all bijections. Then $M$ is a normal submonoid of $A$, since it is the zero-class of the $M$-congruence on $A$,
\[
\forall f,g\in A,\ \forall u\in M,\ (f\circ g\in M \Leftrightarrow f\circ u\circ g\in M).
\]
However, $f\circ M = M\circ f$ does not hold for every $f\in A$: it suffices to take $f$ to be a constant map.
\end{example}

Finally, let $\mathcal{D}_h$ be the full subcategory of $\mathcal{D}_r$ (or $\mathcal{D}_l$) consisting of all pairs $(A,M)$ such that $M$ is a group. These are precisely the preordered monoids for which the preorder defined by the positive cone $M$ is a congruence on $A$.

Indeed, if $M$ is a group and $a\leq_r b$, then $b=ua$ for some $u\in M$, and hence $a=u^{-1}b$, so $b\leq a$. Conversely, if the preorder is symmetric, then $u\geq 0$ implies $u\leq 0$, so there exists $v\in M$ such that $0=vu$, showing that every $u\in M$ is invertible. 
 Thus, $\mathcal{D}_h$ is a subcategory of congruences in monoids. 
 
 Another approach to the study of congruences in monoids is presented in~\cite{Elgueta} where the classical equivalence between congruences in groups and normal subgroups is extended to the case of arbitrary monoids.

\section*{Declarations of interest}

The authors declare no competing interests.

\section*{Funding}
The second author (N.~Martins-Ferreira) acknowledges the Portuguese Foundation for Science and Technology (Fundação para a Ciência e Tecnologia) FCT/MCTES (PIDDAC) through the following Projects: Associate Laboratory ARISELA/P/0112/2020; UIDP/04044/2020; UIDB/04044/2020; PAMI - ROTEI-
RO/0328/2013 (N° 022158); MATIS (CENTRO-01-0145-FEDER-000014 - 3362); PRR-C05-i03-I-000251 (Fruit-PV); Generative.Thermodynamic; by CDRSP and ESTG from the Polytechnic of Leiria.

The third author acknowledges financial support by the Center for Mathematics of the University of Coimbra (CMUC, https://doi.org/10.54499/UID/00324/2025) under the Portuguese Foundation for Science and Technology (FCT), Grants UID/00324/2025 and UID/PRR/00324/2025.

\newpage 

\bigskip
\noindent
Michael Hoefnagel \\ 
\noindent
Department of Mathematical Sciences\\
Stellenbosch University\\
Private Bag X1, Matieland 7602\\
South Africa\\
\quad \\ 
National Institute for Theoretical and Computational Sciences (NITheCS)\\
South Africa \\ 
\texttt{mhoefnagel@sun.ac.za}

\bigskip
\noindent
Manuela Sobral\\
\noindent
Department of Mathematics\\
University of Coimbra\\
Apartado 3008\\
3001--454 Coimbra\\
Portugal\\
\texttt{sobral@mat.uc.pt} \\ 

\bigskip 
\bigskip
\noindent
Nelson Martins-Ferreira\\
\noindent
Departamento de Matemática\\
Escola Superior de Tecnologia e Gestão\\
Politécnico de Leiria\\
2411--901 Leiria\\
Portugal\\
\texttt{martins.ferreira@ipleiria.pt}

\end{document}